\documentclass[a4paper,11pt]{amsart}
\usepackage{amsmath,amssymb,amsfonts,amsthm}
\usepackage{amssymb}
\usepackage{pgf}
\usepackage{todonotes}
\usepackage{epsfig}
\usepackage{mathrsfs}
\usepackage{verbatim}
\newtheorem{defn}{Definition}[section]
\newtheorem{thm}{Theorem}[section]
\newtheorem{prop}{Proposition}[section]
\newtheorem{cor}{Corollary}[section]

\newtheorem{lma}{Lemma}[section]

\newtheorem{exm}{Example}[section]
\def\N{{\rm I\kern-0.16em N}}
\def\R{{\rm I\kern-0.16em R}}
\def\E{{\rm I\kern-0.16em E}}
\def\P{{\rm I\kern-0.16em P}}
\def\F{{\rm I\kern-0.16em F}}
\def\B{{\rm I\kern-0.16em B}}
\def\C{{\rm I\kern-0.46em C}}
\def\G{{\rm I\kern-0.50em G}}
\newcommand{\ud}{\mathrm{d}}
\newcommand{\Q}{\mathbb{Q}}
\newcommand{\ov}[1]{\overline{#1}}
\newcommand{\zz}{\mathbf{z}}
\newcommand{\dd}{\mathbf{d}}
\newcommand{\rr}{\mathbf{r}}
\renewcommand{\ll}{\mathbf{l}}

\numberwithin{equation}{section}
\usepackage{ae}
\def\ind{\mathrel{\hbox{\rlap{%
\hbox to 7.5pt{\hrulefill}}\raise6.6pt\hbox{\eka\char'167}}}}
\begin{document}
\title[Pricing European options on multiple assets]{Multidimensional Breeden-Litzenberger representation for state price densities and static hedging}

\author[Talponen and Viitasaari]{Jarno Talponen \and Lauri Viitasaari}
\address{Department of Physics and Mathematics, University of Eastern Finland, P.O. Box 111, 80101 JOENSUU, talponen@iki.fi}
\address{Department of Mathematics and System Analysis, Helsinki University of Technology\\
P.O. Box 11100, FIN-00076 Aalto,  FINLAND} 

\begin{abstract}
In this article, we consider European options of type \\
$h(X^1_T, X^2_T,\ldots, X^n_T)$ depending on several underlying assets. We study how such options can be valued in terms 
of simple vanilla options in non-specified market models. We consider different approaches related to static hedging and derive several pricing formulas for a wide class of 
payoff functions $h:\R_+^n\rightarrow \R$. We also give new relations between prices of different options both in one dimensional and multidimensional case.
 \end{abstract}


\maketitle

\section{Introduction}
Option valuation is one of the most central problems in financial mathematics. In many models of interest the option valuation cannot be performed in closed form and therefore different approaches have been developed. For instance one can use partial differential equations (PDE) or partial integro-differential (PIDE) methods, Monte Carlo methods, or tree methods (see e.g. \cite{Jeanblanc}, \cite{Oksendal} and \cite{Fengler}). 
One approach to value complicated structured products is to determine their values in terms of the values of simple derivatives of the underlying such as call options, digital options, and, more theoretically, Arrow-Debreu securities. We will study continuous portfolios of these securities and this is essentially static hedging.
In the work of Breeden and Litzenberger \cite{Breeden} it was shown that if the second derivative of the call option price $V^C(K)$ with respect to the strike exists and is continuous, then the price of European option with payoff $f(X_T)$ is given by
\begin{equation}
\label{BL}
V^f = \int_{0}^\infty f(a)\frac{\ud^2}{\ud a^2}V^C(a)\ud a
\end{equation}
where we treat the deterministic short interest rate as $0$ for the sake of simplicity.
Thus the second derivative of the price of the call with respect to the strike price is the state price density of the underlying asset $X_T$. This result has significant applications especially to static hedging which is a field of active research. For more details and discussion, see for instance Carr \cite{Carr}, \cite{Carr2} and references therein.

Bick \cite{Bick} extended the result of Breeden and Litzenberger to a case where either the payoff function or the price
of a call has continuous second derivative with respect to its strike price except in a finite set of points $(s_k)_{k=0}^N$ in which
the left- and right derivatives exist and are finite. In particular, Bick showed that
\begin{equation}
\label{kaava_bick}
\begin{split}
 V^f &= B_T^{-1}f(0) + \int_0^{\infty} f''(a)V^C(a)\ud a\\
&+ B_T^{-1}\sum_{k=0}^{N}\Delta_-f(s_k) \Q(X_T\geq s_k)\\
&+ B_T^{-1}\sum_{k=0}^{N}\Delta_+f(s_k) \Q(X_T> s_k)\\
&+ \sum_{k=0}^N (f'(s_k+)-f'(s_k-))V^C(s_k),
\end{split}
\end{equation}
where $B_T$ denotes the bond function, $\Q$ is the given pricing measure and $\Delta_-$ and $\Delta_+$ denotes the jump of the 
payoff function $f$. For later studies on the relation between European call options and European style derivatives with more general payoff profiles, see also Jarrow \cite{Jarrow}, who derived a characterisation theorem for the distribution function of the underlying asset, and Brown and Ross \cite{Brown}, who consider a model with finite state space and showed that a wide class of options are a portfolio of call options with different strike prices. 
In similar spirit, Cox and Rubinstein \cite{Cox} introduced a method for approximating continuous functions with piecewise linear functions, which are 
a portfolio of call options with different strikes. They also considered the pricing error of this approximation, and suggested that one should find approximation
which is the best in the sense of maximum absolute difference. However, this may cause problems when considering
infinite state space.

Recently the results of in \cite{Breeden} and \cite{Bick} have been extended by the second named author \cite{viitasaari} to cover the case 
where $f$ is only once piecewise differentiable. In particular, in \cite{viitasaari} it is shown that 
\begin{equation}
 \label{v_general}
\begin{split}
 V^f &= B_T^{-1}\E_\Q[f(X_T)]\\
 &= B_T^{-1}f(0) - \int_0^{\infty} f'(a)V^C(\ud
a)\\
&+ B_T^{-1}\sum_{k=0}^{N}\Delta_-f(s_k)\Q(X_T\geq s_k)\\
&+ B_T^{-1}\sum_{k=0}^{N}\Delta_+f(s_k)\Q(X_T> s_k),
\end{split}
\end{equation}
where the measure $V^C(\ud a)$ always exists since $V^C(a)$ is decreasing function in strike. Barrier-type options were also considered in this context. While this result may not be always best option for pricing, it can be used to obtain new theoretical results. For instance, the formula was applied by second named author in \cite{viitasaari2} to study rate of convergence of prices in model approximation. As particular example, new proof and results for rate of convergence of binomial approximation in Black-Scholes model was considered.

To summarize, there exists a vast array of studies on the relation between call and digital option values and values for more general options, in the 
spririt of the formula \eqref{kaava_bick}. However, all the mentioned studies consider market models with bond and one stock and similar results in multidimensional case are not so well-known.

In this article we give under some natural assumptions a pricing formula similar to (\ref{v_general}) for European options $h(\ov{X}_T)=f(X_T^1,X_T^2,\ldots,X_T^n)$ for a wide class of payoff functions $f$, including the rainbow and basket options. In particular, our results cover all continuous functions $h$ for which the partial derivative $\partial_{\sigma(1)}\ldots\partial_{\sigma(m)}h$ exists in the sense of distributions for every $m=1,\ldots,n$ and every permutation $\sigma$ of integers $\{1,\ldots,n\}$. For options which are not of this form we consider standard mollifying techniques with respect to Lebesgue measure. The benefit of this is that in this case the resulting smooth function does not depend on the underlying asset $\ov{X}_T$ or the particular choice of the measure $\Q$. While the results are again theoretical and may not be best way to price derivatives in practice, they provide insight to the pricing mechanism (cf. \cite{viitasaari2} for one dimensional case).

We also derive different relations between prices of different options both in one dimensional and multidimensional case under assumption that the distribution of the underlying assets are absolutely continuous with respect to Lebesgue measure. The methodology is similar to the authors' previous work \cite{note} in which multidimensional version of Breeden-Litzenberger formula (\ref{BL}) was proved.

The benefit of our results is that they are not model-specific. In particular, we only assume that at least one pricing measure for $\ov{X}_T$ exists. We do not assume that it is unique. Moreover, we consider general underlying assets $\ov{X}_T$. Hence our results are valid in models which may be complete or incomplete, or discrete or continuous in time or the state space. 

The problem of inferring the state-price density from observed prices of the derivatives can also be regarded as an inverse problem. One plausible approach would be interpreting the pricing functional as a rather general integral operator 
\[\Phi(f,\Q)=\int f\ d\Q\]
which may be invertible on the latter coordinate if a sufficiently wide class of payoff functions $f$ is included. Instead of inverting the operator forcibly, e.g. by discretizing the operator and then inverting the resulting matrix numerically, we will apply some subtle properties of the payoff function class in question.

The rest of the paper is organised as follows. In Section \ref{sec:pre} we introduce our notation and assumptions. Section \ref{subsec:density} is devoted to path-dependent options and multidimensional versions of Breeden and Litzenberger result for absolutely continuous pricing measures.
In section \ref{sec:main} we present our results involving 
multidimensional Breeden-Litzenberger representation for rather general measures. In subsection \ref{subsec:mollifier} we consider approximation of more general payoff functions with mollifiers. In section \ref{sec:unique} we give a result related to partial uniqueness of the pricing measure $\Q$.

\subsection{Preliminaries}
\label{sec:pre}

In a general model the measure $\Q$ is not necessarily absolutely continuous with respect to Lebesgue measure (more precisely, the law of $\ov{X}_T$ under $\Q$). 
Yet, if the payoff function has enough smoothness as it does in many practical cases, one may apply Theorem \ref{thm_prod_function} (see below).

However, typically the state-price density is absolutely continuous with respect to the Lebesgue and then we have nice representations for it.
It was shown by Breeden and Litzenberger \cite{Breeden} that in one dimensional case the risk-neutral density can be obtained by taking the second derivative of the strike price in the call's price functional. In this section we derive similar results for multidimensional case. 

It is perhaps instructive to first observe that the digital options can be priced with rather minimal machinery and considerations. Namely, Lebesgue's monotone convergence theorem combined with \eqref{eq: 1M} (see below) give immediately the fact that the price of a digital option is the first derivative 
of call price with respect to strike.

We omit the interest rate for simplicity, that is, we will assume it to be deterministic and $0$. Denote $R_{+}^n = \prod_{i=1}^n (0,\infty)$.

Let us consider the following path-dependent payoff profiles: 

\[f_{B}(\{S_t\}_{0\leq t\leq T},T,H,K) = 1_{\max_{0\leq t \leq T} S_t \geq H} (S_T - K)^+ \quad \mathrm{(Barrier)},\] 
and $f(T,K)=f_B (\{S_t\}_{0\leq t\leq T},T,0,K)$ the regular European style call option profile.

\[f_{A}(\{S_t\}_{0\leq t\leq T},T,K)=\left( \frac{1}{T}\int_{0}^T S_t dt - K \right)^+ \quad \mathrm{(Asian)},\] 

\[f_{LB} (\{S_t\}_{0\leq t\leq T},T,K)= (\max_{0\leq t \leq T} S_t  - K)^+ \quad \mathrm{(look-back)},\] 

\[f_{CP} (\{S_t\}_{0\leq t\leq T},T,L,H,K) 1_{\int_{0}^T 1_{S_t\leq H}(t)\ dt\geq L}(S_T - K)^+ \quad \mathrm{(Cumulative\ Parisian)},\] 
\[f_{ML} (\{\overline{S}_t\}_{T_0 \leq t \leq T_1},T_0 , T_1,\overline{K})=\max_{T_0 \leq t \leq T_1 , i}\ (S_i (t) - K_i )^+   \quad \mathrm{(Multi-asset\ look-back)},\] 

\[f_{AB} (\{\overline{S}_t\}_{0\leq t \leq T},T,K)=\int_{0}^T \sum_i (S_i (t) -K_i)^+\ dt \quad \mathrm{(Asian\ basket)}.\]

Similarly as above, let $n$ be the number of underlying stocks (in single-asset case $n=1$). 
We assume that there is a measure $\mu << m_{n+1}$ on $[0,T]\times \R_{+}^n$ defined by the condition 
\[\frac{d\mu}{dm_{n+1}}(t,K_1 , \ldots ,K_n )=\frac{\partial^{n}}{dK_1 \ldots dK_n}\Q\left(\bigwedge_i S^{(i)}_t \leq K_i\right).\] 
The $\sigma$-algebra of $\Q$ may, of course, be much finer than that of its `local push-forward' $\mu$.
Note that the paths may still have jumps. Indeed, we only need the distribution to be absolutely continuous. In particular, this is usually the case if the process is a sum of continuous part and a jump part.
We are assuming the existence of such a pricing measure $\Q$, not the uniqueness of pricing measures in general.
The above $\mu$ can be viewed as a control measure.

In what follows we shall assume that $\max_t S^{(i)}_t $ exist a.s. and have a $\Q$ expected value. 

\section{Multidimensional and path-dependent derivatives. Absolute continuity with respect to Lebesgue measure}\label{subsec:density}

Let us begin with a BL/Bick type formulas connecting the prices of barrier and look-back options:
\begin{thm}
Suppose that the law of $\max_t S_t$ under $\Q$ is absolutely continuous. Then
\[\lim_{\Delta K\to 0^+}\left(\frac{\Delta}{\Delta K}V_{f_{B}} (H,K) \bigg\vert_{K=\Delta K}\right) = \frac{\partial}{\partial K} V_{f_{LB}} (K)\bigg\vert_{K=H}\quad a.e.\ K>0\]
where the derivative on the right-hand side exists for a.e. $K>0$.
\end{thm}
\begin{proof}
The proof is based on the facts that 
\[\frac{\Delta}{\Delta K} f_{B}(\{S_t\}_t , H,K) \bigg\vert_{K=\Delta K} \searrow - 1_{\max_t S_t \geq H}, \quad \Delta K\to 0^+\]
and
\[\frac{\Delta}{\Delta K} f_{LB} \searrow \frac{\partial}{\partial K} f_{LB} = - 1_{\max_t S_t \geq  K}, \quad \Delta K\to 0^+ .\]
To be more precise, the derivative is only defined as a right-sided one at $\max_t S_t =  K$ but this does not affect the expectation, since the law of $\max_t S_t$ under $\Q$ is absolutely continuous.
\end{proof}

The above monotone convergence argument provides access essentially to derivative 
$\frac{\partial}{\partial K} V_{f}(K)$ at the limit $K=0^+$, even if the derivatives need not exist for every $K>0$. By abuse of notation we will sometimes denote by 
$\bigg|_{K=0^+}\frac{\partial}{\partial K} V_{f}(K)$ the limits of the above type.

\begin{thm}
We have 
\[V_{f_{LB}}(\{S_t\}_{0\leq t \leq T},T,K)=-\int_{K}^\infty \frac{\partial}{\partial K} V_{f_B}(T,H,K)|_{K=0^+}\ dH\]
where the derivative on the right-hand side exists for all $H$ and $K$ and the limit exist for all $H$.
\end{thm}

\begin{proof}
First observe that  that 
\[\frac{\partial}{\partial K} V_{f_B}(T,H,K)|_{K=0^+}=\frac{\partial}{\partial K} \E_\Q 1_{\max S_t \geq H}\max(S_T - K)^+=-\E_Q 1_{\max S_t \geq H}\]
by looking at the payoff function, applying the linearity of expectation and applying monotone convergence theorem. Here we apply the fact $\Q(S_T=0)=0$.

Then we note that
\begin{multline*}
\int_{K}^\infty \E_Q 1_{\max S_t \geq H}\ dH = \E_\Q \int_{K}^\infty 1_{\max S_t \geq H}(\omega)\ dH \\
= \E_\Q (\max S_t - K)^+ .
\end{multline*}
\end{proof}
\begin{thm}
Assume that the law of $(S_T,\max_t S_t)$ under $\Q$ is absolutely continuous with respect to $m_2$. Consider a continuous payoff profile of the form $f=h(S_T , \max_{0\leq t\leq T}S_t)$. Then
\[V_f = \int_{0}^{\infty} \int_{0}^\infty   -\frac{\partial^3}{\partial H \partial K^2} V_{f_B}(H,K)\bigg\vert_{H=x, K=y}\ h(x,y)\ dx\ dy.\]
\end{thm}
\begin{proof}
The strategy of the proof is as follows.
\[ -\frac{(\partial^+ )^3}{\partial H \partial K^2} f_B (\{S_t\}_t ,H,K) =1_{\max_t S_t =H}1_{S_T = K}. \]
Thus, under suitable conditions involving $\Q$, dominated convergence theorem gives 
\begin{multline*}
 \int_{0}^{\infty} \int_{0}^\infty   -\frac{\partial^3}{\partial H \partial K^2} V_{f_B}(H,K)\bigg\vert_{H=x, K=y}\ h(x,y)\ dx\ dy\\
 = \int_{0}^{\infty} \int_{0}^\infty -\frac{\partial^3}{\partial H \partial K^2} \E_\Q  f_B (\{S_t\}_t ,H,K)\bigg\vert_{H=x, K=y}\ h(x,y)\ dx\ dy\\
 = \int_{0}^{\infty} \int_{0}^\infty -\lim_{\Delta H, \Delta K \to 0} \frac{\Delta^3}{\Delta H \Delta K^2} \E_\Q  f_B (\{S_t\}_t ,H,K)\bigg\vert_{H=x, K=y}\ h(x,y)\ dx\ dy\\
 = \int_{0}^{\infty} \int_{0}^\infty \frac{\Q (x\leq \max_t S_t < x+dx \wedge y\leq S_T < y+dy)}{dx\ dy} h(x,y)\ dx\ dy\\
 = \E_\Q h(\max_t S_t ,S_T ).
\end{multline*}
Above we only require the law of $(S_T,\max_t S_t)$ under $\Q$ to be absolutely continuous with respect to $m_2$. Indeed, then similar considerations as in the previous proofs yield the required 
differentiability and integrability conditions.
\end{proof}
It remains unknown whether there is a clean condition which guarantees that $(S_T,\max_t S_t)$ will be absolute continuous.
We suspect that the following condition suffices for this, namely, that an absolutely continuous control measure $\mu$ exists and that the realizations $S_t$ are continuous at $t=T$ a.s.

We denote by 
\[V_{f|C}=\E_\Q (f|C)\]
conditional prices.

\begin{prop}\label{thm: f_A_uusi}
We have 
\[\frac{\partial }{\partial K} V_{f_A} =-\Q(f_A > 0).\]
Assume that we have right continuous trajectories a.s. Then
\[\frac{\partial}{\partial K} \frac{1}{T(S_T -K)}\frac{\partial^+ }{\partial T} V_{f_A} =-\Q(f_A >0).\]
and
\[\frac{1}{T}\frac{\partial^+ }{\partial T} V_{f_A | f_A >0}=V_{S_T | f_A >0} -K.\] 
\end{prop}
The motivation for considering conditional expectations here is the following: 
If the Asian option happens to be deep in the money,
then conditional risk-neutral expectation with at/in the money conditioning appears a good 
approximation for the actual value of the option.

\begin{proof}
The first claim is an easy adaptation of the usual Bick formula. To check the second one, note that 
\[\frac{1}{T}\frac{\partial^+ }{\partial T} f_A = 1_{f_A>0} (S_T-K)\]
for every right continuous trajectory and the claim follows easily. The last claim is obtained by taking conditional expectation on both sides of the above equality.
\end{proof}

The next theorem says roughly that in a complete market model with all cumulative Parisians the Asians are also included 
(can be hedged).
\begin{thm}
We have
\begin{multline*}
V_{f_A}(1,K)=\lim_{n\to\infty} \sum_{\substack{\ell_k \geq 0\\ \sum_k \ell_k =1}}
\Bigg(\left(\sum_{k=0}^{n^2 -1} \frac{k}{n} \ell_k - K\right )^+  \\
\Q\left(\bigwedge_k \ell_k +\theta_1 \leq \int 1_{\frac{k}{n}\leq S_t \leq \frac{k+1}{n}}(t)\ dm(t)\leq \ell_k +\theta_2\right)\Bigg)
\end{multline*}
where $\theta_1 = -1_{\frac{k}{n}< K}\frac{1}{n}$ and $\theta_2 = 1_{K\leq \frac{k}{n}}\frac{1}{n}$.
Also,
\[V_{f_A|f_A >0}(1,K)=\int_{H,L>0} (H-K)L\  dF(H,L)\]
where the $2$-dimensional generalized Riemann-Stieltjes integral is taken with respect to the integrator 
\[F(H,L)=-\bigg|_{K=0^+}\frac{\partial}{\partial K}V_{f_{CP}|f_A >0}(H,L,K).\]
Moreover, $f_A (T,K)$ are $\sigma(f_{CP}(H,L,K)\colon H,L,K)$-measurable.
\end{thm}

\begin{proof}
To check the first claim, observe that 
\begin{multline*}
\sum_{\substack{\ell_k \geq 0\\ \sum_k \ell_k =1}}\left(\left(\sum_{k=0}^{n-1} k \ell_k - K\right )^+ 
\prod_k 1_{\ell_k +\theta_1 \leq \int 1_{k\leq S_t \leq k+1}(t)\ dm(t)\leq \ell_k +\theta_2 }\right)\leq f_A (1,K)
\end{multline*}
for every realization and the left term tends to $f_A (1,K)$ for every integrable realization, thus a.s. 
By the monotone convergence theorem applied on the risk-neutral expectation we obtain the claim.

To verify the last part, we are required to show that events of the type 
\begin{equation}\label{eq: abcd}
 a \leq \int 1_{c\leq S_t \leq d}(t)\ dm(t)\leq b
\end{equation}
are $\Q$-measurable. Observe that 
\[\Q(\int 1_{S_t \leq H}(t)\ dm(t)\geq L )=\bigg|_{K=0^+} \frac{\partial^+}{\partial K} V_{f_{CP}}.\]
By varying $H$ and $L$ and using complements we can calculate 
\[\Q\left(a \leq \int 1_{c\leq S_t \leq d}(t)\ dm(t)\leq b\right)\]
for any $a,b,c,d>0$. 

The second part of the statement is seen in the same vein, the Stieltjes integral 
is approximated by summing terms of the form
\[(H-K)L\ \Q\left(L_1 \leq \int 1_{H_1 \leq S_t \leq H_2}(t)\ dm(t)\leq L_2\right),\ H\in [H_1, H_2],\ L \in [L_1 , L_2]\]
over a grid $\{[H_i , H_{i+1}]\times [L_i , L_{i+1}] \colon i\in \N\}$. The errors are of the magnitude
$\leq (L\cdot \Delta H + \Delta L\cdot H)\Q([H,H+\Delta H]\times [L, L+\Delta L])$ which can be easily controlled in the generalized Stieltjes integral.
\end{proof}
By studying the above proof we note that we can `almost' price Asians (unconditionally) from cumulative Parisians price data. The obstruction here is that we do not know the $\Q$-probability of the intersections of the events \eqref{eq: abcd}. This problem could be circumvented if we had the price data of basket options
consisting of $n$ Cumulative Parisians for all $n$ and the other relevant parameters.

The following result is a continuous time version of the multi-asset results in \cite{note}.
\begin{thm}\label{thm: AB}
Assume $\mu$ as above. Then
\begin{multline*}
V_{f_{AB}} = \sum_i \int_{0}^{T}\int_{0}^\infty \dots \int_{0}^\infty  (x_i - K_i)^+ \\
\frac{\partial^{n}}{\partial x_1 \ldots \partial x_n}\sum_i \frac{\partial}{\partial x_i}V_{f_{ML}}(t,t+s,(x_i )) \bigg|_{s=0^+}\ \prod_i dx_i \ dt.
 \end{multline*}
 where the partial derivatives exists a.e.
\end{thm}

\begin{proof}[Proof of Theorem \ref{thm: AB}]
The statement reduces to checking that 
\begin{multline*}
\frac{\partial}{\partial T} V_{f_{AB}} =\sum_i \int\ldots \int \frac{d\mu}{dm_{n+1}}(T,x_1 ,\ldots ,x_n )\ (x_i-K_i)^+\ \prod_i dx_i\\
=\sum_i \int_{0}^\infty \dots \int_{0}^\infty   (x_i - K_i)^+ \\
\frac{\partial^{n}}{\partial x_1\ldots \partial x_n}\sum_i \frac{\partial}{\partial x_i}V_{f_{ML}}(T,T+s,(x_i ))\bigg|_{s=0^+}\ \prod_i dx_i 
\end{multline*}
for a.e. $T$. Indeed, we may differentiate the value of $f_{AB}$ with respect to $T$ a.e. by Lebesgue's 
differentiation theorem because we have assumed that $\max_t S^{(i)}_t$ have $\Q$ expectation.
Note that here we are not using the properties of $\mu$.

Next, we use the fact that $\mu|_{t=T} <<m_n$ for a.e. $T$. We observe that 
\[\frac{\partial}{\partial x_i}V_{f_{ML}}(T,T+s,(x_i ))\bigg|_{s=0^+}= \frac{\partial}{\partial x_i}\E_\Q \max_i (S^{(i)}_T - x_i )^+\]
by Lebesgue's monotone convergence theorem and the classical Breeden-Litzenberger result in the 
absolutely continuous case to verify the existence of the partial derivatives.

We will apply the following formula involving fixed time $T$, 
\begin{multline*}
\frac{d \mu}{dm_{n+1}}(T,K^{(1)}, \ldots , K^{(n)} )\\
= \frac{\partial^{n}}{\partial K^{(1)}\ldots \partial K^{(n)}}\sum_i \frac{\partial}{\partial K^{(i)}}\E_\Q \max_i (S^{(i)}_T - K^{(i)})^+ ,
\end{multline*}
see \cite{note}. By collecting together the above observations we conclude that the Radon-Nikodym derivative
of the measure $\mu$ can be recovered as follows:
\begin{multline*}
\frac{d \mu}{dm_{n+1}}(T,K^{(1)}, \ldots , K^{(n)} )\\
=\frac{\partial^{n}}{\partial K^{(1)}\ldots \partial K^{(n)}}\sum_i \frac{\partial}{\partial K^{(i)}}V_{f_{ML}}(T,T+s,(K^{(i)}))\bigg|_{s=0^+}
\end{multline*}
for a.e. $(T,K^{(1)} ,\ldots , K^{(n)})$. 
\end{proof}

We note that in the previous result the term 
\[\bigg|_{s=0^+}\frac{\partial^{n+1}}{\partial s \partial K^{(1)}\ldots \partial K^{(n)}}\sum_i \frac{\partial}{\partial K^{(i)}}V_{f_{ML}}(t,t+s,(K^{(i)}))\]
can be seen as a time-dependent multi-asset state price density. In principle, this can be applied in pricing different types of path dependent derivatives.

\section{Static hedging for pricing measures with jumps}
\label{sec:main}
Let $S_t^k$ denote the stock price processes 
and $X_t^k$ the underlying assets of an option. As examples, $X_t^k$ can be
a functional of $S_t^k$:s like the average $X_t^k = \frac{1}{t}\int_0^t S_u^k\ud u$
representing Asian option, $X_t^k = \max_{u\leq t}S_u^k$ representing Lookback option, $X_t^j =
\max_{1\leq k\leq d}S_t^k$
 representing Rainbow option or $X_t^j = \sum_{k=1}^d \alpha_kS_t^k$ representing Basket
option. Throughout the article, $B_t$ denotes the bond given by an non-decreasing deterministic function with $B_0=1$ (all the results can be extended to stochastic interest rate models, with obvious changes involving $X_{t}^0=r_t$ in the results). 
A vector $(x_1,\ldots,x_n)$ is denoted by $\ov{x}$. Similarly, $\ov{X}_t$ denotes the vector 
$$
\ov{X}_t = (X_t^1,\ldots, X_t^n).
$$

We assume that our model is, to some extent, free of arbitrage which means that 
there exists at least one pricing measure $\Q$ such that for each claim $C$, the discounted value at time $t$ is given by $\E_{\Q}[B_T^{-1}C|\mathcal{F}_t]$. For 
more details on mathematics of arbitrage, see \cite{delbaen} and \cite{schachermayer} and references therein. 
In the notation, we usually omit the dependence on $\Q$ and $\E$ stands for expectation with respect to $\Q$. We also assume that 
for given maturity $T$ we have
$X_T^k\in L^1(\Q)$ and $X_T^k\geq 0$ almost surely. 
Moreover, the price of a European option with payoff profile $h(X^1_T,\ldots,X^n_T)$ is denoted by $V^h$. 
We present our result for prices only i.e. values at time $t=0$. However, our results could be extended to cover 
values at arbitrary time $t$ with obvious changes. Note also that we assume the maturity $T$, but omit it on the notation. 

\begin{defn}
For a function $f:\R_+\rightarrow\R$, we denote $f\in\Pi_\Q(X_T)$ if the following conditions are satisfied:
\begin{enumerate}
\item
$f$ is continuously differentiable except on at most countable set of points
$0\leq s_0 < s_1 <\ldots < s_\alpha < \ldots $ ($\alpha < \gamma$ countable ordinals) in which $f$ and $f'$ have jump-discontinuities,
\item
$f(X_T)\in L^1(\Q)$,
\item
$f$ satisfies
\begin{equation}
\label{cond1}
\lim_{x\rightarrow\infty}|f(x-)|\Q(X_T \geq x) = 0
\end{equation}
and,
\item
the integral
$$
\int_0^{\infty}f'(a)\Q(X_T>a)\ud a
$$
is finite.
\end{enumerate}
\end{defn}
\begin{defn}
We denote by $\mu_{c,-}$ and $\mu_{c,+}$ the weighted counting measures, so that for a given function $f:\R_+\rightarrow\R$, we have
\begin{enumerate}
\item
\begin{equation}
\int_0^x f(y)\ud \mu_{c,-}(y) = \sum_{y\leq x}\tilde{\Delta}_-f(y),
\end{equation}
where $\tilde{\Delta}_-f(y) = f(y)-f(y-)$, 
\item
\begin{equation}
\int_0^x f(y)\ud\mu_{c,+}(y) = \sum_{y< x}\tilde{\Delta}_+f(y),
\end{equation} 
where $\tilde{\Delta}_+f(y)=f(y+)-f(y)$.
\end{enumerate}
The jump from the left at $0$ is defined as $\tilde{\Delta}_-f(0)=0$.
\end{defn}
We will also need the following counting measures.
\begin{defn}
We denote by $|\mu|_{c,-}$ and $|\mu|_{c,+}$ the counting measures such that for a given function $f:\R_+\rightarrow\R$, we have
\begin{enumerate}
\item
\begin{equation}
\int_0^x f(y)\ud |\mu|_{c,-}(y) = \sum_{y\leq x}|\tilde{\Delta}_-f(y)|,
\end{equation}
\item
\begin{equation}
\int_0^x f(y)\ud|\mu|_{c,+}(y) = \sum_{y< x}|\tilde{\Delta}_+f(y)|.
\end{equation} 
\end{enumerate}
\end{defn}
For a given measure $\Q$ and underlying process $\ov{X}_t$ we consider the following class of payoff functions.
\begin{defn}
For a function $h:\R_+^n\rightarrow\R$, we write $h\in\Pi^n_\Q(\ov{X}_T)$ if the following conditions are satisfied:
\begin{enumerate}
\item
\begin{equation}
\label{prod_func}
h(\ov{x}) = \sum_{k=1}^m \prod_{j=1}^n f_{k,j}(x_j),
\end{equation}
where  $f_{k,j}\in \Pi_\Q(X_T^j)$, $k=1,\ldots,m$ and $j=1,\ldots,n$.
\item
for every $k=1,\ldots,m$, every $i=1,\ldots,n$, and every permutation $\sigma =(\sigma(1),\ldots,\sigma(n))$ of integers $1,\ldots,n$ we have
$$
\prod_{j=1}^i f_{k,\sigma(j)}(X_T^{\sigma(j)})\in L^1(\Q).
$$
In particular, $h(\ov{X}_T) \in L^1(\Q)$.
\item
for every $k=1,\ldots,m$ and every $i=1,\ldots,n$
\begin{equation}
\label{limit_assu}
\lim_{b\rightarrow\infty}|f_{k,i}(b-)|\E\left[\textbf{1}_{X_T^i\geq b}\prod_{j=1}^{i-1}|f_{k,j}(X_T^j)|\right]=0.
\end{equation}
\end{enumerate}
\end{defn}

We note that all polynomials are of form (\ref{prod_func}).

We also need some operators for further use. For the rest of the paper $\partial_k$ denotes the usual partial derivative with respect to variable $x_k$. 
We find it convenient to use multi-indices to formulate the main results. 
Recall that a multi-index is a vector 
$\mathbf{a} \in \{0,1,2,\ldots\}^n$ which encodes the order of each pure multiple partial derivative in a mixed higher order partial derivative.
In what follows all multi-indices $\mathbf{a}$ satisfy $\|\mathbf{a} \| :=\max_i a_i \leq 1$, being binary sequences, which means that they can also be regarded as subsets of $\{1,\ldots,n\}$. Recall the following standard notation: $|\mathbf{a} | := a_1 + \ldots + a_n$.

\begin{defn}
For a function $h:\R_+^n\rightarrow\R$ we define operator $\mathbf{0}_k$ by
\begin{equation}
\mathbf{0}_kh(\ov{x}) = h(x_1,\ldots,x_{k-1},0,x_{k+1},\ldots,x_n). 
\end{equation}
\end{defn}

We will apply the multi-index notation for the operators $\mathbf{0}_k$ as well. 
So, for example $\mathbf{0}^{\mathbf{a}} \partial^{\mathbf{b}} h (x_1 , x_2 , x_3 , x_4 )= \partial_3 \partial_4 h(0, 0, x_3 , x_4 )$ for $n=4$, $\mathbf{a}=(1,1,0,0)$ and $\mathbf{b}=(0,0,1,1)$.
One may also consider $e_1 ,e_2 \in \mathbf{a}$; $e_3 , e_4 \in \mathbf{b}$; $e_1 , e_2 , e_3 , e_4 \in \mathbf{a}+\mathbf{b}$.

\begin{defn}
Let $h\in\Pi^n_\Q(\ov{X}_T)$ and let $\zz$, $\dd$, $\rr$, and $\ll$ (short for 'zero', 'derivative', 'right jump' and 'left jump', respectively) be multi-indices such that 
$|\zz+ \dd + \rr + \ll| =n$ and $\|\zz + \dd + \rr + \ll\| = 1$. We consider a functional $A_{\zz,\dd,\rr,\ll} :\Pi^n_\Q(\ov{X}_T)\rightarrow\R$ (implicitly depending on $\Q$ and $\ov{X}_T$) given by
\begin{multline*}
A_{\zz,\dd,\rr,\ll} (h) = \int_{\R_+^{n-|\zz|} } \mathbf{0}^{\zz}  \partial^{\dd} h(\ov{y})\\
\Q\left(\left(\bigwedge_{\sigma\in \dd +\rr} (X_T^{\sigma}> y_{\sigma})\right)\bigwedge \left(\bigwedge_{\sigma\in \ll} (X_T^{\sigma} \geq y_{\sigma})\right)\right)\\
\prod_{\sigma \in \dd}\ud y_{\sigma} \prod_{\sigma\in \rr} \ud\mu_{c,+}(y_{\sigma}) \prod_{\sigma\in \ll}\ud\mu_{c,-}(y_{\sigma}) .
\end{multline*}
We will also require a similar positive functional:
\begin{multline*}
|A|_{\zz ,\dd ,\rr ,\ll}= \int_{\R_+^{n-|\zz|} } |\mathbf{0}^{\zz}  \partial^{\dd} h(\ov{y})|\\
\Q\left(\left(\bigwedge_{\sigma\in \dd + \rr} (X_T^{\sigma}> y_{\sigma})\right)\bigwedge \left(\bigwedge_{\sigma\in \ll} (X_T^{\sigma} \geq y_{\sigma})\right)\right)\\ 
\prod_{\sigma \in \dd}\ud y_{\sigma} \prod_{\sigma\in \rr} \ud|\mu|_{c,+}(y_{\sigma}) 
\prod_{\sigma\in \ll}\ud|\mu|_{c,-}(y_{\sigma}) .
\end{multline*}
\end{defn}

For further use we also consider restrictions of operators $A_{\zz,\dd,\rr,\ll}$ and $|A|_{\zz,\dd,\rr,\ll}$ to a subset 
$K^{n-|\zz|}\subset \R_+^{n-|\zz|}$ i.e.
\begin{multline*}
A_{\zz,\dd,\rr,\ll}^K (h) = \int_{K^{n-|\zz|}} \mathbf{0}^{\zz}  \partial^{\dd} h(\ov{y})\\
\Q\left(\left(\bigwedge_{\sigma\in \dd +\rr} (X_T^{\sigma}> y_{\sigma})\right)\bigwedge \left(\bigwedge_{\sigma\in \ll} (X_T^{\sigma} \geq y_{\sigma})\right)\right)\\
\prod_{\sigma \in \dd}\ud y_{\sigma} \prod_{\sigma\in \rr} \ud\mu_{c,+}(y_{\sigma}) \prod_{\sigma\in \ll}\ud\mu_{c,-}(y_{\sigma}),
\end{multline*}
and $|A|_{\zz,\dd,\rr,\ll}$ is defined similarly.

The definition of the operators admittedly appears complicated at first sight. However, it is rather natural, we simply start with a function $h$ and choose $|\zz|$ variables which we set to zero. Next we choose $|\dd|$ variables and compute partial derivatives with respect to these variables. 
Next we choose $|\rr|$ variables and consider right jumps with respect to these variables and for the remaining $|\ll|$ variables we consider left jumps. Finally, we weight the resulting function with probability where for partial derivatives and right jumps we consider strict tails and for left jumps we 
consider tail probabilities of the form $\Q(X_T^m\geq y_m)$, and integrate over these variables. 
The functional $A_{\zz , \dd , \rr , \ll}$ computes this and later on we may sum over every possible permutation. Moreover, if $|A|_{\zz ,\dd ,\rr ,\ll}$ is finite,
then also $A_{\zz ,\dd , \rr ,\ll}$ is finite and well-defined. The following result can be seen as a multidimensional version of Bick's representation 
\eqref{kaava_bick}.

\begin{thm}
\label{thm_prod_function}
Let $h\in\Pi^n_\Q(\ov{X}_T)$. If 
\begin{equation}
\label{prod_assumption}
|A|_{\zz , \dd , \rr , \ll} (h) < \infty
\end{equation}
for each combination of multi-indices $\zz$, $\dd$, $\rr$, and $\ll$ such as above, then the price of a European option with payoff $h(\ov{X}_T)$ is given by
\begin{equation}
V^h = B_T^{-1}\sum_{\underset{\|\zz + \dd + \rr + \ll\|=1}{|\zz + \dd + \rr + \ll |=n}}A_{\zz , \dd , \rr, \ll} (h).
\end{equation}
\end{thm}

We note that according to the above result we may price option of given type as follows: for every variable $y_k$ we either set it to zero, take partial derivative or consider jump from right or left and then integrate with respect to corresponding measure. The price is obtained by summing over all possible combinations. 
As a result we obtain $4^n$ terms. However, usually payoff functions are continuous at least with respect to some of the variables. Hence many of the terms vanish.

\begin{exm}
As an example set $n=2$ and consider up-and-in Barrier call option with strike $K$ and barrier $H$ given by
$$
f(S_T, X_T) = (S_T - K)^+\mathbf{1}_{X_T \geq H},
$$
where $X_T = \max_{0\leq u\leq T}S_u$. Note that now we have only one underlying but we can treat $X_T$ as another underlying. The price of this option is given by
$$
V^f = B_T^{-1}\int_K^\infty \Q(S_T>y \wedge X_T\geq H)\ud y.
$$
This result is already established in \cite{viitasaari}.
\end{exm}
\begin{exm}
As a continuous example in case $n=2$, consider rainbow options with payoff
\begin{equation*}
h_p(x, y,K_1, K_2,, K) = \left( \left(((x-K_1)^+)^p + ((y-K_2)^+)^p
\right)^{1/p}-K\right)^{+},
\end{equation*}
where $0<p<\infty$. For $p=1$ the price of this option with underlyings $X_T$ and $Y_T$ is given by
\begin{equation*}
\begin{split}
V^{h_1} &= B_T^{-1}\int_{K_1+K}^\infty \Q(X_T > z)\ud z + B_T^{-1}\int_{K_2+K}^\infty \Q(Y_T > z)\ud z\\
& + B_T^{-1}\int_{K_2}^{K+K_2} \Q(Y_T >z \wedge X_T > K_1 + K_2 + K -z)\ud z.
\end{split}
\end{equation*}
By differentiating prices $V^{h_p}$ one can obtain multidimensional Breeden-Litzenberger formula (see \cite{note} for details).
\end{exm}

The proof of Theorem \ref{thm_prod_function} is based on the following lemma which is an extension of results in \cite{viitasaari}. The proof is presented 
in the appendix. 
\begin{lma}
\label{lma:apu}
Let $f\in\Pi_\Q(X_T)$, and $Y\in L^1(\Q)$ such that
\begin{equation}
\label{lemma_assu_1}
\int_0^{\infty} |f'(a)|\E[\mathbf{1}_{X_T>a}|Y|]\ud a< \infty,
\end{equation}
\begin{equation}
\label{lemma_assu_2}
\int_0^{\infty}|f(a)|\E[\mathbf{1}_{X_T>a}|Y|]\ud|\mu|_{c,+}(a)<\infty,
\end{equation}
\begin{equation}
\label{lemma_assu_3}
\int_0^{\infty}|f(a)|\E[\mathbf{1}_{X_T\geq a}|Y|]\ud|\mu|_{c,-}(a) < \infty,
\end{equation}
and
\begin{equation}
\label{limit_lemma_assu}
\lim_{b\rightarrow\infty}|f(b-)|\E[\textbf{1}_{X_T\geq b}|Y|]=0. 
\end{equation}
Then
\begin{equation}
 \label{v_general2}
\begin{split}
\E[f(X_T)Y] &= f(0)\E[Y] + \int_0^{\infty} f'(a)\E[\mathbf{1}_{X_T>a}Y]\ud a\\
&+ \int_0^{\infty}f(a)\E[\mathbf{1}_{X_T>a}Y]\ud\mu_{c,+}(a)\\
&+ \int_0^{\infty}f(a)\E[\mathbf{1}_{X_T\geq a}Y]\ud\mu_{c,-}(a).
\end{split}
\end{equation}
\end{lma}
\begin{proof}[Proof of Theorem \ref{thm_prod_function}]
By linearity it is sufficient to consider function
$$
h(\ov{x}) = \prod_{k=1}^n f_k(x_k).
$$
We put $Y_1=\prod_{k=1}^{n-1}f_k(X_T^k)$ and apply Lemma \ref{lma:apu} for payoff  $f_n(X_T^n)Y_1$ to obtain
\begin{equation*}
\begin{split}
 V_t^f &= B_T^{-1}f_n(0)\E[Y_1] - B_T^{-1}\int_0^{\infty} f'_n(a)\E[\mathbf{1}_{X_T^n>a}Y_1]\ud a\\
&+ B_T^{-1}\int_0^{\infty}f_n(a)\E[\mathbf{1}_{X_T^n>a}Y_1]\ud\mu_{c,+}(a)\\
&+ B_T^{-1}\int_0^{\infty}f_n(a)\E[\mathbf{1}_{X_T^n\geq a}Y_1]\ud\mu_{c,-}(a).
\end{split}
\end{equation*}
Now we can compute $\E[\mathbf{1}_{X_T^n>a}Y_1]$ (the term $\E[\mathbf{1}_{X_T^n\geq a}Y_1]$ is treated similarly) by setting 
$Y_2 = \mathbf{1}_{X_T^n>a}\prod_{k=1}^{n-2}f_k(X_T^k)$ and applying Lemma \ref{lma:apu} for $f_{n-1}(X_T^{n-1})Y_2$. Indeed, assumption (\ref{limit_assu}) implies that (\ref{limit_lemma_assu}) is satisfied for every $Y_i$. Moreover, (\ref{prod_assumption}) implies that assumptions (\ref{lemma_assu_1})-(\ref{lemma_assu_3}) are satisfied for every $Y_i$. Hence, by proceeding similarly and applying Lemma \ref{lma:apu} repeatedly, we obtain the result. 
\end{proof}

In many practical cases the payoff function is continuous but not of form (\ref{prod_func}). By taking a sequence of functions $h_k\in\Pi^n_\Q(\ov{X}_T)$ we obtain similar results for limiting functions having enough smoothness. 
For discontinuous functions the jump parts may cause problems. However, we can approximate discontinuous functions with continuous ones. This is the topic of the next subsection.

Note also that results provided in this section are closely related to static hedging which is particularly interest when considering transaction costs and specification errors. In particular, if digital multi-asset options are traded in the market then more general options can be (approximately) hedged with digital options. This comes particularly clear in a market with one asset in which more general options can be hedged with digital options on that particular asset. Moreover, if the payoff has additional smoothness (which it usually has) such as second derivative almost everywhere then one can hedge the payoff with call options (for details in one dimension see e.g. \cite{Carr}).

\subsection{Pricing with distributions}
Recall that for each continuous functions $h:\R_+^n\rightarrow\R$ all the mixed partial derivatives $\partial^{\beta} h$
exist in the sense of distributions, see \cite{Hormander} for discussion. Therefore for every continuous function $g$ with compact support there exists a sequence of smooth 
(test) functions $h_n$, obtained by applying the Stone-Weierstrass Theorem on compact sets, such that
\[ \int_{\R_+^n}g(\ov{y})\partial^{\beta} h(\ud\ov{y}) = \lim_n \int_{\R_+^n}g(\ov{y})\partial^{\beta} h_n(\ov{y})\ud\ov{y}.\]

The order of taking the partials does not matter in the above formula because of the possibility of approximating with polynomials.
Thus each of the functionals $A_{\zz , \dd , \rr ,\ll}:\Pi^n_\Q(\ov{X}_T)\rightarrow\R$ can be uniquely extended to $\tilde{A}_{\zz,\dd,\rr,\ll}$ 
with the range of all continuous functions $h$. Indeed, if $h$ is continuous, we set $\tilde{A}_{\zz,\dd,\rr,\ll}= 0$ whenever $\rr \neq 0$ or $\ll\neq 0$. 
The functionals $\tilde{A}_{\zz,\dd,0,0}$ are defined naturally since partial derivatives of $h$ exist.
Evidently the same condition holds for $|\tilde{A}|_{\zz,\dd,\rr,\ll}$.

\begin{thm}
Let $h:\R_+^n\rightarrow\R$  be a continuous payoff function such that
\begin{equation}
\label{prod_assumption2}
|A|_{\zz,\dd,0,0}(h) < \infty
\end{equation}
for all multi-indices $|\zz + \dd|=n$, $\|\zz+\dd\|=1$. Then the price of a European option $h(\ov{X}_T)$ is given by
\begin{equation}
V^h = B_T^{-1}\sum_{\underset{\|\zz+\dd\|=1}{|\zz + \dd|=n}} A_{\zz,\dd,0,0}(h).
\end{equation}
\end{thm}
\begin{proof}
Assume first that $h\in C_0^{\infty}(\R_+^n)$ and let $N$ be a number such that $supp(h)\subset[0,N]^n$. From real analysis we know that in a compact set $[0,N+\delta]^n$ we can approximate $h$ with a sequence of polynomials $T_n(\ov{x})$ uniformly such that the partial derivatives of $T_n$ convergence to partial derivatives of $h$ also. Now, setting $T_n(x)=0$ outside $[0,N+\delta]^n$, by using suitable coordinate-wise convolution, we have $T_n(\ov{x})\in\Pi^n_\Q(\ov{X}_T)$ for every $n$. Hence the claim for $h\in C_0^{\infty}(\R_+^n)$ follows. 

Assume next that $h$ is merely continuous. By assumption (\ref{prod_assumption2}) there exists a compact set 
$K\subset\R_+^{n-i}$, a finite union of suitable smaller compact sets $K_{\zz,\dd}$, such that
\[\sum_{\underset{\|\zz+\dd\|=1}{|\zz + \dd|=n}} \left|\tilde{A}_{\zz,\dd,0,0}(h)\right| < \epsilon,\]
where
\begin{equation*}
\tilde{A}_{\zz,\dd,0,0}(h) = \int_{\R_+^{n-|\zz|}\backslash K_{\zz\dd}} \mathbf{0}^{\zz}\partial^{\dd} h(\ov{y}) \Q\left(\bigwedge_{\sigma\in \dd}(X_T^{\sigma}> y_{\sigma})\right)\prod_{\sigma\in \dd} \ud   y_{\sigma}.
\end{equation*}
Since $h$ is continuous, we can take a sequence $h_k\in C_0^{\infty}$ such that $h_k$ converges to $h$ uniformly on compact sets and all the partial derivatives converge in the space of distributions $\mathcal{D}' (\R_{+}^n)$ accordingly. Thus 
\[A^{K_{\zz,\dd}}_{\zz,\dd,0,0}(h_k) \to A^{K_{\zz,\dd}}_{\zz,\dd,0,0}(h)\]
as $k\to \infty$ for all multi-indices $|\zz + \dd|=1$, $\|\zz +\dd\|=1$. Hence the claim follows as $\epsilon>0$ was arbitrary.
\end{proof}
\begin{exm}
Consider a spread option $f(X_T^1,X_T^2)=(X_T^1-X_T^2)^+$. Now $f(x,y)$ is continuous, $f(0,0)=f_y(0,y)=0$, and $f_x(x,0)=1$. Moreover, $f_{xy}(x,y)$ exists in the sense of distributions and equals $-\delta_x(y)$, where $\delta_x(y)$ is the Dirac delta function at $x$. Hence we obtain 
\begin{equation*}
\begin{split}
V^f &= B_T^{-1}\int_0^\infty \Q(X_T^1>y)\ud y - B_T^{-1}\int_{\R_{+}^2} \Q(X_T^1>x \wedge X_T^2>y)\ud x \delta_x(\ud y)\\
&= B_T^{-1}\E[X_T^1] - B_T^{-1}\int_0^\infty \Q(X_T^1>y \wedge X_T^2>y)\ud y.
\end{split}
\end{equation*}
\end{exm}
\begin{exm}
Consider a payoff function $f(x,y)=\mathbf{1}_{x\geq y}$. Now $f$ is not continuous nor of form (\ref{prod_func}). However, we have
$$
f(x,y)= \lim_{\epsilon\rightarrow 0+} \frac{(x-y+\epsilon)^+ - (x-y)^+}{\epsilon}.
$$
Hence, by the previous example and the dominated convergence theorem, we may calculate formally as follows
\begin{equation*}
\begin{split}
\E[f(X_T^1,X_T^2)] &= -\lim_{\epsilon\rightarrow 0+}\int_0^\infty \frac{\Q(X_T^1 >y \wedge X_T^2>y+\epsilon) - \Q(X_T^1>y \wedge X_T^2>y)}{\epsilon}\ud y\\
&= -\bigg|_{x=0^+}\frac{\partial}{\partial x}\int_0^\infty \Q(X_T^1>y\wedge  X_T^2 > y+x)\ \ud y.
\end{split}
\end{equation*}
Another way to price such options is to use mollifiers as in next subsection.
\end{exm}
\subsection{Approximation with smooth functions}
\label{subsec:mollifier}
Our main theorem explains how options with sufficient smoothness can be priced.
In this section we consider general integrable functions $h$ and consider how to apply our results for pricing such options. We use mollifiers with respect to the Lebesgue measure. The benefit is that this mollifier does not depend on $\ov{X}$, 
and hence not on the particular choice of $\Q$. 

We use the standard mollifier given by
\begin{equation}
\label{mollifier}
\rho(\ov{x}) = c\mathbf{1}_{|\ov{x}|<1}e^{-\frac{1}{1-|\ov{x}|^2}},
\end{equation}
where $|\cdot|$ denotes the standard Euclidean norm and $c$ is a constant such that
$$
\int_{\R^n}\rho(\ov{x})\ud\ov{x}=1.
$$
Now we have $\rho\in C_0^{\infty}$. Let now $h$ be an arbitrary function. Formally we set
\begin{equation}
h_{\epsilon}(\ov{x}) = \int_{\R^n}\rho(\ov{y})h(\ov{x}-\epsilon\ov{y})\ud\ov{y}.
\end{equation}
A standard result of real analysis states that for sufficiently small $\epsilon$, $h_{\epsilon}$ is infinitely differentiable on compact subsets. Moreover, if $h$ is continuous, then $h_{\epsilon}\rightarrow h$ uniformly on compact subsets.
We also recall the following fact from real analysis (cf. \cite{Rudin}). 
\begin{lma}
\label{lma:cont_appro}
Let $\mu$ be a positive Radon measure on $\R^n$. For any $h\in L^1(\R^n, \mu)$ and any $\epsilon>0$, there exists a function $\varphi\in C_0(\R^n)$ such that
$$
||h-\varphi||_{L^1 (\mu)}< \epsilon.
$$
\end{lma}
We now proceed to consider our case.
\begin{thm}
Assume that $h(\ov{X}_T)\in L^1(\Q)$. Then the following are equivalent:
\begin{enumerate}
\item
\begin{equation}
\label{mollifier_convergence}
V^{h_{\epsilon}}\rightarrow V^h,
\end{equation}
\item
\begin{equation}
\label{strong_convergence}
\E|h_{\epsilon}(\ov{X}_T) - h(\ov{X}_T)| \rightarrow 0,
\end{equation}
\item
for every $\tilde{\epsilon}$ there exists a compact set $K$ and a constant $\eta>0$ such that
\begin{equation}
\label{mollifier_UI}
\sup_{0<\epsilon<\eta}|\E[h_{\epsilon}(\ov{X}_T)]\mathbf{1}_{\ov{X}_T\in\R_+^n\backslash K}| < \tilde{\epsilon}.
\end{equation}
\end{enumerate}
\end{thm}
\begin{proof}
Without loss of generality we can omit the bond. 

$(3)\Rightarrow(2)$: Assume we have (\ref{mollifier_UI}) and fix $\tilde{\epsilon}$. By (\ref{mollifier_UI}) we 
can take compact subset $K\subset \R_+^n$ such that
$$
\E|h_{\epsilon}(\ov{X}_T) - h(\ov{X}_T)|\mathbf{1}_{\ov{X}_T\in\R^n\backslash K} < \frac{\tilde{\epsilon}}{4}.
$$
Moreover, by Lemma \ref{lma:cont_appro} we can take continuous $\varphi$ such that 
$$
\E|h(\ov{X}_T) - \varphi(\ov{X}_T)|<\frac{\tilde{\epsilon}}{4}.
$$
We obtain
\begin{equation*}
\begin{split}
&\E|h_{\epsilon}(\ov{X}_T) - h(\ov{X}_T)|\\ =& \E|h_{\epsilon}(\ov{X}_T) - h(\ov{X}_T)|\mathbf{1}_{\ov{X}_T\in K} + \E|h_{\epsilon}(\ov{X}_T) - h(\ov{X}_T)|\mathbf{1}_{\ov{X}_T\in\R^n\backslash K}\\ 
\leq & \E|h_{\epsilon}(\ov{X}_T) - \varphi_{\epsilon}(\ov{X}_T)|\mathbf{1}_{\ov{X}_T\in K} + \E|\varphi_{\epsilon}(\ov{X}_T) - \varphi(\ov{X}_T)|\mathbf{1}_{\ov{X}_T\in K}\\
+& \E|\varphi(\ov{X}_T) - h(\ov{X}_T)|\mathbf{1}_{\ov{X}_T\in K} + \E|h_{\epsilon}(\ov{X}_T) - h(\ov{X}_T)|\mathbf{1}_{\ov{X}_T\in\R^n\backslash K}.
\end{split}
\end{equation*}
Since $\varphi$ is continuous, the second term is bounded by $\frac{\tilde{\epsilon}}{4}$ for sufficiently small $\epsilon$. To finish the proof we obtain by continuity of $\varphi$, compactness of $K$ and assumption (\ref{mollifier_E_bounded}) that
$$
 \E|h_{\epsilon}(\ov{X}_T) - \varphi_{\epsilon}(\ov{X}_T)|\mathbf{1}_{\ov{X}_T\in K} < \frac{\tilde{\epsilon}}{4}.
$$
$(2)\Rightarrow(1)$: this implication is obvious.\\
$(1)\Rightarrow(3)$: Assume that we have $(1)$ and (\ref{mollifier_UI}) does not hold. Let $\tilde{\epsilon}>0$ be fixed. 
Since $h$ is integrable, we can find compact set $K$ such that
\begin{equation*}
\E|h(\ov{X}_T)|\mathbf{1}_{\ov{X}_T\in\R_+^n\backslash K} < \frac{\tilde{\epsilon}}{2}
\end{equation*}
Moreover, by $(1)$ we have
\begin{equation*}
\E|h_{\epsilon}(\ov{X}_T) - h(\ov{X}_T)|\mathbf{1}_{\ov{X}_T\in\R_+^n\backslash K}<\frac{\tilde{\epsilon}}{2}
\end{equation*}
for $\epsilon$ sufficiently small. Hence we also have
\begin{equation*}
\E|h_{\epsilon}(\ov{X}_T)|\mathbf{1}_{\ov{X}_T\in\R_+^n\backslash K} < \tilde{\epsilon},
\end{equation*}
and we have a contradiction. This completes the proof.
\end{proof}

Note that condition $(3)$ is closely related to the notion of uniform integrability with respect to pair $(B,h)$, 
a notion which was
introduced in \cite{viitasaari2} in one-dimensional case.

In many practical case the Condition (\ref{mollifier_UI}) is satisfied. The next result 
gives an easily verifiable sufficient condition, under which we have (\ref{mollifier_UI}). 
\begin{cor}
If for every $\tilde{\epsilon}>0$ there exist a compact set $K\in\R_+^n$ and a constant $\eta$ such that 
\begin{equation}
\label{mollifier_E_bounded}
\sup_{0<\epsilon<\eta}\textit{ess}\sup_{|\ov{y}|\leq 1}|\E[h(\ov{X}_T - \epsilon\ov{y})]\mathbf{1}_{\ov{X}_T\in\R_+^n\backslash K}| < \tilde{\epsilon},
\end{equation}
Then we also have (\ref{mollifier_UI}).
\end{cor}

Many options of interest have payoff functions which are polynomially bounded, i.e.
there exists a polynomial $p(\ov{x})$ such that
$$
|h(\ov{x})|\leq p(\ov{x}), \forall\ov{x}.
$$
For example this is the case for standard rainbow options. Now
$$
\E[p(\ov{X}_T)]<\infty
$$ 
implies that we also have (\ref{mollifier_E_bounded}).

\section{On the uniqueness of arbitrage-free prices}
\label{sec:unique}
The main theme in this paper has been deducing the pricing kernel from observed prices of European style options with several underlying assets. 
If there is an explicit formula for the pricing measure, then, of course, the measure must be unique. One may ask if the price information about some class of 
derivatives sufficiently determines the measure, although no explicit formula may not be available (cf. \cite{Brown}).

The following fact is probably evident to the specialists in the field and it can be obtained rather immediately from the considerations in the previous section.

\begin{prop}
The $n$-asset pyramid option values $C(K_1,\ldots,K_n,0)$ determine the joint pricing law $\Q$ uniquely if $\Q$ is absolutely continuous with respect to the Lebesgue measure on the state space, $S_i$ are independent and the interest rate is deterministic.
\end{prop}

Here the $n$-asset pyramid payoff has the form 
\[h(\overline{S},\overline{K},K)=\left(\sum (S_i - K_i)^+  - K\right)^+ .\]

\begin{proof}
This fact is based on disintegrating the measure similarly as in \cite{note}. 
\end{proof}

Next we will attempt to see how some partial information obtained from the prices of a class of derivatives translates to a kind of partial uniqueness of the pricing measure.
The information accumulation can be conveniently encoded in terms of sub-$\sigma$-algebras, as is customary in the probability theory. 
 
The following result roughly states that derivatives obtained by multiplying European call options determine the pricing measure on the sub-$\sigma$-algebra they generate.

\begin{thm}
Let $\mathcal{F}$ be a set of non-negative Borel functions on the state space $\R_{+}^n$ considered as European style payoff. Let $\mathcal{Q}$ be a set of Borel probability 
measures on the same state space. Assume that $\E_\Q \prod_{i=1}^{n}(f_i -K_i)^+$ exists and does not depend on the particular choice of $\Q\in \mathcal{Q}$ 
for 
$K_1 , \ldots ,K_n \in \R_+$ and $f_1,\ldots, f_n \in \mathcal{F}$.
Suppose that $g$ is a $\sigma(\mathcal{F})$-measurable payoff function. Then $g$ has $\Q$-expectation 
if and only it has expectation with respect to all the measures in the family. 
Moreover, the value $\E_\Q g$, when defined, does not depend on the particular choice of $\Q\in \mathcal{Q}$.
\end{thm}

Recall that $\sigma(\mathcal{F})$ is the smallest $\sigma$-algebra containing the sets $f^{-1}(U)$ for $f\in \mathcal{F}$ and $U\subset \R$ open and
the statement that $g$ is $\Sigma$-measurable means that $g^{-1}(U)\in \Sigma$ for each open $U\subset \R$.The payoff profiles $f_i$ and $g$ can be seen
as random variables with respect to the different measures $\Q\in \mathcal{Q}$. For example putting 
$g=\E_{\Q_0}(h|\sigma(\mathcal{F}))$ gives a typical $ \sigma(\mathcal{F})$-measurable
random variable. We will apply Dynkin's lemma which is well suited for analyzing the uniqueness of measures, see \cite{Halmos}.

\begin{proof}
We denote by $\mathcal{D}$ the collection of all Borel sets $A$ such that $\Q(A)$ does not depend on the particular choice of $\Q$.
Note that $\mathcal{D}$ is closed with respect to taking complements, since $\Q$ are probability measures.
By the $\sigma$-additivity of the measures we observe that 
$\mathcal{D}$ is closed with respect to taking countable unions of disjoint subsets. Thus 
$\mathcal{D}$ is a Dynkin system.

We note that $\sigma(\mathcal{F})$ is generated by the sets $f\geq K$ with $f\in \mathcal{F}$, $K\geq 0$. 
In what follows we restrict our attention to $\sigma(\mathcal{F})$, i.e. 
the measures and measurable functions are considered with this $\sigma$-algebra.

We aim to show that $\sigma(\mathcal{F})\subset \mathcal{D}$. This suffices in order to obtain the statement for $\sigma(\mathcal{F})$-simple 
functions $g$. Indeed, then it follows that the measures in $\mathcal{Q}$ restricted to $\sigma(\mathcal{F})$ coincide. 
If $g$ should be integrable with respect to a measure $\Q_0 \in \mathcal{Q}$, then $\E_{\Q_0}(g)$ can be 
approximated by expectations of simple functions $\E_{\Q_0}(g_n)$ with $g_n \nearrow g$ $\Q_0$-a.s. as $n\to \infty$. 
It follows that the values $\E_{\Q}(g_n)$ coincide for different choices of $\Q$. Note that $g_n \nearrow g$ $\Q$-a.s. as 
$n\to \infty$ by the equivalence of the measures. Thus, $\E_{\Q}(g)=\E_{\Q_0}(g)$ by the Monotone Convergence 
Theorem. 

Next we invoke Dynkin's Theorem, which yields that if the sets of the form  
\begin{equation}\label{eq: pi}
\bigwedge_{i=1}^{n} M_i \leq f_i < K_i  
\end{equation}
are included in $\mathcal{D}$, then also $\sigma(\mathcal{F})\subset \mathcal{D}$. Indeed, the collection of 
sets in \eqref{eq: pi} are closed with respect to taking finite intersections and they also $\sigma$-generate sets 
$M <f_i < K$ and consequently the $\sigma$-algebra $\sigma(\mathcal{F})$ as well. 

Recall that an indicator function $1_{M_i\leq f_i < K_i}$ can be written as $1_{M_i \leq f_i} -1_{K_i \leq f_i}$ and that 
\begin{equation}\label{eq: 1M}
1_{M_i \leq f_i}=\lim_{\epsilon \to 0^+}\frac{(f_i - M_i + \epsilon)^+ - (f_i - M_i)^+}{\epsilon}.
\end{equation}
This means that $\prod_{i=1}^{n} 1_{M_i \leq f_i < K_i}$ can be written as $\lim_{\epsilon \to 0^+}g_{\epsilon}$ where
\[g_{\epsilon}=\prod_{i=1}^{n}\frac{(f_i - M_i + \epsilon)^+ - (f_i - M_i)^+ - (f_i - K_i + \epsilon)^+ 
+ (f_i - K_i)^+}{\epsilon}.\]
Note that by the assumptions and the linearity of taking expectations, the value $\E_{\Q}(g_\epsilon )$, when defined, 
does not depend on $\Q$. Since $\prod_{i=1}^{n} 1_{M_i \leq f_i < K_i}=\lim_{\epsilon\to 0^+}g_\epsilon$ we obtain by applying the 
Monotone Convergence Theorem for each measure $\Q$ separately that 
\begin{equation}\label{eq: Q}
\Q(\bigwedge_{i=1}^{n} M_i \leq f_i < K_i)=\lim_{\epsilon\to 0^+}\E_{\Q}(g_\epsilon)
\end{equation}
does not depend on $\Q$. Indeed, the functions $g_\epsilon$ have expectations according to the assumptions. 
Thus the sets \eqref{eq: pi} are included in $\mathcal{D}$. 
\end{proof}

\vskip0.5cm

\section{Conclusions}
In this paper we have proved several new theoretical results related to option prices. Firstly, we have studied connections between prices of different options. The importance of such connections is that one can price more complex options with simpler ones if one can find prices for latter options. Secondly, we have provided pricing formulas for options depending on multiple assets with a general payoff. While our results are purely theoretical they can be applied to different areas of mathematical finance. For example, one can use our results to study approximation results in multidimensional case. Moreover, options depending on multidimensional assets are not yet so vastly studied in the literature. In this and the previous article \cite{note} we have provided general framework to study options on many assets.

\textbf{Acknowledgements}.\\ 
Lauri Viitasaari thanks the Finnish Doctoral Programme in Stochastics and Statistics for financial support. 

\appendix

\section{Proof of Lemma \ref{lma:apu}}
Proof is based on following lemmas. 
\begin{lma}
\label{lma:apu_multi}
Assume that 
$$
\prod_{k=1}^m X_T^{\sigma(k)}\in L^1(\Q)
$$
for every $m=1,\ldots,n$ and every permutation $\sigma$. Then the expected value of
$$
h(\ov{X}_T;\ov{y}) = \prod_{k=1}^n(X_T^k-y_k)^+
$$
is absolutely continuous with respect to Lebesgue measure.
\end{lma}
\begin{proof}
The claim follows directly from observation that 
$$
|(X_T^k-y_k)^+ - (X_T-y_k-h)^+| \leq h.
$$
\end{proof}
\begin{lma}
\label{lemma_v_general}
Let $\alpha\geq 0$, $\alpha<\beta<\infty$ and consider a function $g$ of 
the form $g(x)=f(x)\textbf{1}_{\alpha\leq x\leq \beta}$, where $f$ is continuous on
$[\alpha,\beta]$ and continuously differentiable on $(\alpha,\beta)$. If $Y\in L^1$, then
\begin{equation}
\begin{split}
\E[g(X_T)Y] &= f(\alpha)\E[\mathbf{1}_{X_T\geq\alpha}Y] -
f(\beta)\E[\mathbf{1}_{X_T>\beta}Y]\\
&+ \int_\alpha^\beta f'(a) \E[\mathbf{1}_{X_T>a}Y]\ud a.
\end{split}
\end{equation}
\end{lma}
\begin{proof}
The proof is essentially the same as presented in \cite{viitasaari} for case $Y=1$. Indeed,
since $g$ is continuous on $[\alpha,\beta]$, we can approximate it with 
\begin{equation}
\label{appro1}
\begin{split}
g_n(x) &= f(\alpha)\textbf{1}_{x=\alpha} + \sum_{k=1}^n (c_kx + b_k)\textbf{1}_{a_k < x \leq a_{k+1}}\\
&= f(\alpha)\textbf{1}_{x=\alpha}+\sum_{k=1}^n (c_kx + b_k)(\textbf{1}_{a_k < x}-\textbf{1}_{a_{k+1}<x}),
\end{split}
\end{equation}
where $\alpha = a_1 < a_2 < \ldots < a_{n+1} = \beta$ is a partition of the interval
$[\alpha,\beta]$ and
the
coefficients are given by
$$
c_k = \frac{f(a_{k+1}) - f(a_k)}{a_{k+1} - a_k},
$$
and
$$
b_k = f(a_{k+1}) - c_ka_{k+1} = f(a_k) - c_ka_k.
$$
Simple computations and taking expectation yields 
\begin{equation*}
\begin{split}
\E[g_n(X_T)Y] &= f(\alpha)\E[\mathbf{1}_{X_T\geq\alpha}Y] -
f(\beta)\E[\mathbf{1}_{X_T>\beta}Y]\\
&+ \sum_{k=1}^n c_k \left[\E[(X_T-a_k)^+Y] - \E[(X_T-a_{k+1})^+Y]\right].
\end{split}
\end{equation*}
Now applying Mean value theorem, continuity of $f'$, and the fact that $\E[(X_T-a)^+Y]$ is a monotone function with respect 
to $a$ we obtain
\begin{equation*}
\sum_{k=1}^n c_k \left[\E[(X_T-a_k)^+Y] - \E[(X_T-a_{k+1})^+Y]\right] \rightarrow \int_\alpha^{\beta} f'(a)\ud \E[(X_T-a)^+Y].
\end{equation*}
Moreover, by Lemma \ref{lma:apu_multi} we have
$$
\int_\alpha^{\beta} f'(a)\ud \E[(X_T-a)^+Y]=\int_\alpha^{\beta} f'(a)\E[\textbf{1}_{X_T>a}Y]\ud a.
$$
It remains to note that $g_n$ converges to $g$ pointwise. Hence the result follows by Dominated convergence theorem.
\end{proof}
\begin{lma}
\label{rema_apu}
Let $\alpha\geq 0$, $\alpha<\beta<\infty$ and consider a function $g^0$ of 
the form $g^0(x) = f(x)\textbf{1}_{\alpha< x < \beta}$, where $f$ is continuous on
$[\alpha,\beta]$ and continuously differentiable on $(\alpha,\beta)$. If $Y\in L^1$, then
\begin{equation}
\begin{split}
\E[g^0(X_T)Y] &= f(\alpha)\E[\mathbf{1}_{X_T\geq\alpha}Y] -
f(\beta)\E[\mathbf{1}_{X_T>\beta}Y]\\
&+ \int_\alpha^\beta f'(a)\E[\mathbf{1}_{X_T>a}Y]\ud a.
\end{split}
\end{equation}
\end{lma}
\begin{proof}
Define a new function $g$ by
\begin{equation*}
g(x) = \begin{cases}
f(\alpha+),& x = \alpha\\
f(x),& x \in (\alpha,\beta)\\
f(\beta-),& x = \beta.
\end{cases}
\end{equation*}
By Lemma \ref{lemma_v_general}, we 
obtain
\begin{equation*}
\begin{split}
\E[g(X_T)Y]&= f(\alpha)\E[\mathbf{1}_{X_T\geq\alpha}Y] -
f(\beta)\E[\mathbf{1}_{X_T>\beta}Y]\\
&+\int_\alpha^\beta f'(a)\E[\mathbf{1}_{X_T>a}Y]\ud a.
\end{split}
\end{equation*}
Noting that $g^0(x) = g(x) - g(\alpha)\textbf{1}_{x=\alpha}-g(\beta)\textbf{1}_{x=\beta}$ we obtain the result.
\end{proof}
\begin{proof}[Proof of Lemma \ref{lma:apu}]
Put
$g_b(x) = f(x)\textbf{1}_{0\leq x < b}$, and set $s_{n+1}=b$.
By assumptions, we may write
$$
g_b(x) = \sum_{k=0}^n f(x)\textbf{1}_{s_k<x<s_{k+1}} + \sum_{k=0}^n f(x)\textbf{1}_{x=s_k}.
$$
Applying Lemma \ref{rema_apu} for terms on the first sum and direct computations yields
\begin{equation*}
\begin{split}
\E[f(X_T)Y] &= f(0)\E[Y] + \int_0^{b} f'(a)\E[\mathbf{1}_{X_T>a}Y]\ud a\\
&+ \int_0^{b}f(a)\E[\mathbf{1}_{X_T>a}Y]\ud\mu_{c,+}(a)\\
&+ \int_0^{b}f(a)\E[\mathbf{1}_{X_T\geq a}Y]\ud\mu_{c,-}(a)\\
&- f(b-)\E[\textbf{1}_{X_T\geq b}Y].
\end{split}
\end{equation*}
Letting $b$ tend to infinity and applying Dominated convergence result together with assumptions gives the result.
\end{proof}

\end{document}